\documentclass{amsart}
\usepackage{amsmath,amscd,amsfonts,amssymb,amsthm, multirow, newlfont,exscale,color,mathrsfs, multicol}

\title{Universal Lex Ideal Approximations \\ of Extended Hilbert Functions and  \\ Hamilton Numbers}
\author{Tigran Ananyan and Melvin Hochster$^1$}
\date{\today}
\usepackage{enumerate}
\usepackage{mathtools}
\numberwithin{equation}{section} 
\newtagform{show_eq}{(Eq.\ }{)}  
\usetagform{show_eq}
\usepackage{bm}
\theoremstyle{plain}
\newtheorem{theorem}{Theorem}[section]

\newtheorem{corollary}[theorem]{Corollary}

\newtheorem{proposition}[theorem]{Proposition}
\newtheorem{lemma}[theorem]{Lemma}

\theoremstyle{remark}
\newtheorem{remark}[theorem]{Remark}
\newtheorem{notation}[theorem]{Notation}
\newtheorem{discussion}[theorem]{Discussion}
\theoremstyle{definition}

\newcommand{\inc}{\subseteq}

\newcommand{\cL}{\mathcal{L}}

\newcommand{\fA}{\mathfrak{A}}

\newcommand{\fB}{\mathfrak{B}}

\newcommand{\fL}{\mathfrak{L}}

\newcommand{\imp}{\Rightarrow}

\newcommand\HN{H^{(N)}}

\newcommand{\N}{\mathbb{N}}

\newcommand{\Q}{\mathbb{Q}}

\newcommand\RN{R^{(N)}}
\newcommand\Rg[1]{R^{(#1)}}

\newcommand{\Z}{\mathbb{Z}}

\newcommand\vect[2]{#1_1,\,\ldots,\, #1_{#2}}

\def\vect#1#2{{#1}_1, \, \ldots, \, {#1}_{#2}}

\def\rbx#1#2{\raisebox{#1 pt}{\hbox{${#2}$}}}

\def\ovadj#1#2#3#4{$\overset{\rbx{#1}{#3}}{\rbx{#2}{#4}}$}

\newcommand{\mx}{\begin{pmatrix}}
\newcommand{\emx}{\end{pmatrix}}

\def\bcp#1#2{\mathbf{{\biggl\langle}}\, \raisebox{-4.5pt}{\ovadj {7} {-3.6} {#1} {#2}}\, \mathbf{{\biggr\rangle}}}
\def\dsp{\displaystyle}
\def\bco#1#2{\dsp {{#1} \choose {#2}}}
\def\todo#1
{\par\vskip 7 pt \noindent {\textcolor{yellow}\footnotesize \color{yellow} \fbox{\parbox{4.8in}{\color{black}
\textbf{To do: } {\color{blue}#1}}}} \vskip 3 pt \par}
\def\forth#1
{\par\vskip 7 pt \noindent {\textcolor{yellow}\footnotesize \color{yellow} \fbox{\parbox{4.8in}{\color{black}
\textbf{For thought: } {\color{blue}#1}}}} \vskip 3 pt \par}

\begin{document}

\begin{abstract} Let $\Rg h$ denote the polynomial ring in variables $\vect x h$ over a specified field $K$.  
We consider all of these rings simultaneously, and  in each use lexicographic (lex) monomial order with $x_1 > \cdots > x_h$.
 Given a fixed homogeneous ideal $I$ in $\Rg h$, for each $d$ there is unique lex ideal generated
in degree at most $d$ whose Hilbert function agrees with the Hilbert  function of $I$ up to degree $d$.  When we
consider $I\Rg N$ for $N \geq h$,  the set $\fB_d(I,N)$ of minimal generators for this lex ideal in degree at most $d$ may change,
but $\fB_d(I,N)$ is constant  for all $N \gg 0$. We let $\fB_d(I)$ denote the set of generators one obtains for all $N \gg 0$, and we
let $b_d = b_d(I)$ be its cardinality.    The sequences $\vect b d, \, \ldots$ obtained in this way may grow very fast.  Remarkably, even when
$I = (x_1^2, x_2^2)$, one  obtains a very  interesting sequence,  0, 2, 3, 4, 6, 12, 924, 409620,$\,\ldots$.  This
sequence is the same as  $H_{d-1} + 1$ for $d \geq 2$,  where $H_d$ is the  $d\,$th  
Hamilton number.   The Hamilton numbers were studied by Hamilton and by 
Hammond and Sylvester  because of their occurrence in a counting problem connected with the use of Tschirnhaus
transformations in manipulating polynomial equations. \end{abstract}
 
\subjclass[2010]{Primary 13P10, 13D40}

\keywords{polynomial ring, form, Hilbert function, lex ideal, regular sequence, Hamilton numbers}

\thanks{$^1$The second author was partially supported by National Science Foundation grants
DMS--1401384 and DMS--1902116.}

\maketitle

\pagestyle{myheadings}
\markboth{TIGRAN ANANYAN AND MELVIN HOCHSTER}{UNIVERSAL LEX APPROXIMATIONS  AND HAMILTON NUMBERS}

\section{Introduction}\label{intro}  
  
There has been a great deal of work recently (see, for example, \cite{AH1, AH2, AH3, CCMPV, Dr, DrLL, ESS, HuMMS1, HuMMS2, MaMc} 
and their references) on the behavior of invariants of ideals generated in at most
a given degree and with at most a given number of generators when the number of variables is not bounded in any way.
The problems we study were motivated by a related question described below.  Suppose we fix the Hilbert function of a homogeneous  ideal $I$ in a polynomial ring over a field $K$ in several variables.  An example would be to fix the Hilbert function of the ideal generated by $x_1^2$, $x_2^2$.  It is easy to see that the Hilbert functions of  all the extensions of $I$ to polynomial rings in more variables, where the number of variables $N$ may be arbitrarily large, are determined by the original Hilbert function.  We think of the Hilbert functions of $I$ and its extensions as a model for the behavior of a large class of ideals.  In our example,  
the Hilbert function of the ideal generated by a regular sequence consisting of any two quadratic forms, not necessarily $x_1^2,\,x_2^2$, is obtained once the number of variables is sufficiently large.  

Throughout this paper, the terms ``lexicographic" and ``lex" are used interchangeably in describing both monomial
orders and ideals.  In dealing with lexicographic order, we always assume $x_1 >  x_2 > \cdots > x_N$.  
A monomial ideal $I$ will be called {\it lexicographic} or {\it lex}  (some authors use the term {\it lexsegment}) if whenever $\mu'  > \mu$ 
are monomials of  degree $d$ and $\mu \in I$ then $\mu' \in I$.  

Suppose that an ideal $I_N$ in $\RN := K[\vect x N]$ has the same Hilbert function as $I\RN$.  
When one calculates a Gr\"obner basis, with respect to any monomial order, for  $I_N$, one obtains an initial ideal, a monomial ideal with the same Hilbert function.  The length of the calculation depends, of course, on the choice of $I$, and can grow with $N$, but an upper bound  can be predicted if one has an upper bound for the number of minimal generators of the initial ideal.    By  a result of Macaulay \cite{Mac}  the lex ideal with the same Hilbert function has the largest number 
of minimal generators of any homogeneous ideal with the same Betti numbers.   This is generalized in \cite{Big, Hul1, Par2}, where it this shown that
the lex ideal with a given Hilbert function has the largest graded Betti numbers of any homogeneous  ideal with that Hilbert function.     
We were therefore led to explore this question:   when one considers the extension of $I$ to a polynomial ring in $N$ variables, where $N \gg 0$, how many generators  are needed in each degree $d$ for the lex ideal with the same Hilbert function as the extension of $I$? 
 The existence of such a lex ideal was shown by Macaulay \cite{Mac}.  It turns out that for any specific degree, the 
 number of minimal generators needed  in that degree is constant for all $N \gg 0$.   See Discussion~\ref{mainrec}.
For other background in this area see     \cite{Hul2, IP, Mer, MerP1, MerP2, Par1, Sn1, Sn2, Sn3}.  A general treatment 
of monomial ideals is given in \cite{HeHi}.  

Our results give bounds on the lengths of Gr\"obner basis calculations, for any monomial order, for ideals of a certain ``shape," where 
the shape is described in terms of an ideal in a relatively small number of variables.  Fix the field $K$. 
We shall say that a finitely generated graded module $M_1$ over $R^{(h_1)}$ has the same {\it shape} as a finitely generated graded
module $M_2$ over $R^{(h_2)}$ if $M_1$ has the same graded Betti numbers as $M_2$.   If we extend $M_1$ to
a polynomial ring in more variables by a degree-preserving map, the shape is preserved.  The shape, together
with the number of variables, determines the Hilbert function. 

We illustrate this idea
by focusing on the case of a regular sequence of quadratic forms of length two, which is an example of
a simple shape, coming from the ideal $(x_1^2,\,x_2^2) \inc K[x_1,x_2]$.    For all regular sequences of 
quadratic forms of length 2 generating an ideal $I$,  the
modules $R/I$ have the same graded Betti numbers.

 Other examples are a regular sequence of length $h$ of forms of specified degrees
$\vect d h$  or $t \times t$ minors of an $r \times s$ matrix of forms $F_{ij}$ of degrees $d_{ij}$ such
that the $t$ size minors are homogeneous and the height of the ideal is the same as in the generic case: in this
case, the shape may vary with the characteristic of the base field (cf.~\cite{Ha}).

In the case of a regular sequence of quadratic forms of length two, suppose that we have such a sequence
in the polynomial ring in $N$ variables, and we want to bound the length of a Gr\"obner basis calculation, with 
respect to some monomial basis,  for the ideal they generate.  Such a bound can be determined from the least number of 
generators of the  initial ideal one obtains, and the number of generators is largest for a lex ideal whose Hilbert function
is the same as the one obtained from the two quadratic forms,  by Macaulay's result or the results of \cite{Big, Hul1, Par2} cited above.
Similarly, in general, the numbers of generators of the various lex ideals with the same Hilbert
function, once the ideal is extended to $\RN$, give information bounding the length
 of the Gr\"obner basis calculation for ideals of a given shape as the number of variables $N$ grows.  
  
 In the case of our example of a regular sequence consisting of two quadratic forms, our detailed results are given
 in Theorems~\ref{ham} and~\ref{lham} in \S\ref{2quform}. 
 
We now make all this precise, and describe in detail the main question that we study.   
Throughout this paper, let $K$  be a fixed, specified, but arbitrary field $K$. 
Let $I$ be a specified ideal in $\Rg h$.  Observe that if we want to study homogeneous regular sequences 
$\vect F n$ with respective degrees
$\vect m n$, we may take  $I = (x_i^{m_i}:1 \leq i \leq n)\Rg h.$  
For all $N \geq h$, let $\cL_{I,N}$ denote the unique lex ideal with the same
Hilbert function as $I\Rg N$.  In Discussion~\ref{mainrec}, we shall see that the set $\fA_d(I,N)$ (respectively,  
$\fB_d(I,N)$) of minimal generators of $\cL_{I,N}$ that have degree 
$d$ (respectively, degree at most $d$)  is constant for all $N \gg 0$, and we let $\fA_d(I)$ (respectively, $\fB_d(I)$)
denote the stable value.  Let $a_d(I)$ and $b_d(I)$ denote the respective cardinalities of the sets $\fA_d(I)$  and
$\fB_d(I)$.  Clearly,  $b_d(I)$ is the sum of the values $a_j(I)$ for  $j \leq d$.   Our primary objective is the study of the sequences 
of numbers $a_d(I)$ and $b_d(I)$.  One of our main results is a recursive formula for $b_{d+1}$ in the general case: see 
Theorem~\ref{main}. 

When $I$ is generated by a regular sequence consisting of two quadratic forms,for $d \geq 2$ the sequence $b_d(I)$ coincides with the sequence 
$H_{d-1}+1$,  where $H_d$ is $d\,$th Hamilton number, studied by Hamilton and others in a completely different context:   see \S\ref{Ham}.
These numbers have double exponential growth (cf.~Proposition~\ref{Hamprop}(f)).

When a lex ideal is extended to a polynomial ring in more variables, it may fail to be lexicographic. For example,  $(x_1^2, x_1x_2, x_2^2)$ 
is lexicographic in $\Rg 2$ but not in $\Rg 3$, because in the larger ring $x_1x_3 > x_2^2$.  We shall say that an ideal of 
$\Rg h$ is a {\it universal lex} ideal if its extension to $\Rg N$ is lexicographic for all $N \geq h$.  The usage of
{\it universal} is the same as in \cite{HeHi}. 
This notion will be quite  useful
in understanding the behavior of $a_d(I)$ and $b_d(I)$.  Universal lex ideals are studied
in detail in \S\ref{stablex}.

\begin{notation}\label{bcop} We need two kinds of notation connected with binomial coefficients.   
We use $\dsp {n \choose r}$ with its standard
meaning for $n \geq r \geq 0$, and it is defined to be zero if  $n$ or $r$ is negative or $r > n$.  If $r \geq 0$, we write 
$\bcp n r$ for  
 $\dsp {1\over r!}\prod_{j = 0}^{r-1}(n-j)$
for all $n \in \Z$, where the product is $1$ if $r = 0$.     For fixed $r \geq 0$,  this is the unique
polynomial in $n$ of degree $r$ that agrees with $\dsp {n \choose r}$ for all $n \geq r$. 
In fact, it is easy to see that $\dsp {n \choose r} = \bcp n r$ whenever $n \geq 0$.  

Note that $\bcp n 0 = 1$
for all $n \in \Z$,  that  $\bcp n r = 0$ for $0 \leq n < r$,  and that for $n > 0$,  
\begin{equation}\label{bc1} 
\bcp {-n} r = (-1)^r {n+r-1 \choose r}. 
\end{equation}
Also note that for all $n \in \Z$ and all $r \geq 1$,
\begin{equation}\label{bc2} \bcp n r = \bcp {n-1} r + \bcp {n-1} {r-1},
\end{equation}
since this well-known identity for binomial coefficients holds for $n \geq r$ and both sides are polynomials in $n$.
We note that the usual binomial coefficient identities
\begin{equation}\label{bc3}\hbox{(i)\ }\bco n r = \bco {n-1} r + \bco {n-1} {r-1} \hbox{\ \ and\ \ } \hbox{(ii)\ }\bco n r - \bco {n-1}{r-1} = \bco n{r-1}
\end{equation}
hold for all $n,\,r \geq 1$,  even if $r$ is large, so that a numerator may exceed a denominator.

We make frequent use of the fact that for $N > 0$,  the Hilbert function of  $\Rg N$ is 
$t \mapsto \bco{N+t-1}{N-1}$ for all $t \in \Z$.  Note that this not correct when $N = 0$:  there is one
incorrect value, when  $t = 0$, since $\bco{-1}{-1}$ is 0,  not 1.
\end{notation}

\section{Universal lex ideals}\label{stablex}  

We shall show that universal lexicographic ideals are very constrained.  
\begin{notation}\label{notGam}

 Suppose that we have a universal lexicographic ideal and
that it has minimal generators in degrees $d_1 < d_2 < \cdots < d_h = d$ where $d_1 \geq 1$, and that the number of minimal
generators in degree $d_i$ is $\alpha_i \geq 1$.  We shall let $\Gamma$ denote the set of pairs $\{(d_j, \alpha_j)\}$.
There is no constraint on the number of variables, and we may
assume that $N$ variables, where $N \gg 0$, are available.  This information uniquely determines the generators
and the ideal,  since in choosing the next minimal generator to use in given degree, it must be the largest monomial of
that degree that is not a multiple of any previously chosen monomials.  Let $\fL_{\Gamma, N}$ denote this universal lexicographic
ideal whenever $N$ is sufficiently large.
 It will be convenient to define $\beta_0 := 0$ and 
$\beta_i := \alpha_1 + \alpha_2 + \cdots + \alpha_i$, $1 \leq i \leq h$, so that Theorem~\ref{stablexdescrip} below gives a complete, explicit 
description of the set  of minimal generators once $\Gamma$ is specified.  
\end{notation}

\begin{notation}\label{abnot} For this same ideal and for $i \in \N$, we can define $a_i$ and $b_i$ in $\N$ as follows.  
We let $a_i := \alpha_j$ if $i$ is one of the $d_j$, and $a_i :=0$ otherwise.  We define 
$b_i$ as the sum of the $\alpha_j$ such that $d_j \leq i$, which is the same as the largest of the $\beta_j$ for $j \leq i$.  Note that
$b_i$ is also the sum of the $a_j$ such that $j \leq i$.  Thus,  $a_i$ is the number of minimal generators of the ideal in 
degree $i$,  and $b_i$ is the number of minimal generators of the ideal of degree at most $i$.
\end{notation}

Since the quotient of a polynomial ring in $M$ variables by an ideal generated by  $m$ variables is polynomial ring 
in $M-m$ variables, we have the following preliminary result:

\begin{lemma}\label{Mm} The Hilbert function of the ideal generated by $m$ variables in a polynomial ring in $M > m$ variables over a field
is $t \mapsto \bco {M + t - 1}{M-1} - \bco {M-m +t -1}{M-m-1}$. \qed \end{lemma}

\begin{theorem}\label{stablexdescrip} Consider the universal lexicographic ideal $\fL_{\Gamma, N}$ corresponding to $\Gamma$ as described in
Notation~\ref{notGam}. 
Let $\mu_1 := x_1^{d_1-1}$ and, recursively, if  $1 \leq j \leq h-1$,  
$\mu_{j+1} := \mu_jx_{\beta_j +1}^{d_{j+1}-d_j}$. Thus, 
$\mu_2 = x_1^{d_1-1}x_{\beta_1 +1}^{d_2-d_1}$, and
$\mu_h = x_1^{d_1-1}x_{\beta_1+1}^{d_2-d_1}\cdots x_{\beta_{h-1}+1}^{d_h - d_{h-1}}$.
 Then, in degree $d_j$, the minimal generators of $\fL_{\Gamma, N}$ will be 
$$(\dagger_j) \quad\mu_j x_{\beta_{j-1}+1}, \mu_j x_{\beta_{j-1}+2},  \ldots, \mu_j x_{\beta_{j-1}+i}, \ldots, \mu_j x_{\beta_{j-1}+\alpha_j},$$
where the last term may also be written as $\mu_j x_{\beta_j}$. \smallskip 

We let $\fB_\Gamma$ denote this set of generators. 
The number of generators of this ideal, is evidently $\beta_h$,  which is also the total number of variables occurring in the
generators. For any $N \geq \beta_h$, $\fL_{\Gamma, N}$ is the ideal generated by $\fB_\Gamma$ in $\Rg N$.
  \smallskip

For any fixed $N \geq \beta_h$, the set of monomials in $\fL_{\Gamma,N}$  that are multiples of a minimal generator of degree $d_j$
but not of a minimal generator of degree $d_i$ for $i < j$  is the same as the set of monomials in 
$$\mu_j (x_{\beta_{j-1}+1}, x_{\beta_{j-1}+2},  \ldots, x_{\beta_{j-1}+i}, \ldots, x_{\beta_{j-1}+\alpha_j})S_j$$
where $S_j$ is the polynomial ring generated over $K$ by all of the consecutive variables $x_{\beta_{j-1}+1}, \, \ldots \, x_N$,
so that   $S_j$ is a polynomial ring in $N-\beta_{j-1}$ variables over $K$. 
\smallskip

Hence, for $N > \beta_h$,  the Hilbert function of the ideal $\fL_{\Gamma,N}$ is 
$$\,\,\,\,(\dagger) \qquad t \mapsto \sum_{j=1}^h \bigg(\bco{N-\beta_{j-1}+ t-d_j}{N-\beta_{j-1}-1} - \bco{N-\beta_j+t-d_j}{N-\beta_j-1}\biggr) \quad\hbox{\rm or} $$
$$ {}\!\!\!\!(\dagger') \qquad t \mapsto \sum_{j=1}^h \bigg(\bco{N-\beta_{j-1}+ t-d_j}{t-d_j +1} - \bco{N-\beta_j+t-d_j}{t - d_j + 1}\biggr).$$
\end{theorem}

Before giving the proof, we make a remark and exhibit the first three cases without using the $\beta$ notation.

\begin{remark} Requiring that $N > \beta_h$ rather than $N \geq \beta_h$ is only needed because we need to make sure that
the formula for the Hilbert function for the term subtracted is correct 
when $j = h$,  $N = \beta_h$,  and $t = d_j$.  In this case the term subtracted should be $1$, not $0$. \end{remark}

In degree $d_1$ we must first have $x_1^{d_1}$, and then $x_1^{d_1-1}x_2, \ldots, x^{d_1-1}x_{\alpha_1}$ as the minimal generators. 
In degree $d_2$ we must  have $$x_1^{d_1-1}x_{\alpha_1+1}^{d_2-d_1+1}, x_1^{d_1-1}x_{\alpha_1+1}^{d_2-d_1}x_{\alpha_1 + 2}, \ldots
 x_1^{d_1-1}x_{\alpha_1+1}^{d_2-d_1}x_{\alpha_1 + \alpha_2}$$ as the new minimal generators.
 
 In degree $d_3$ we must  have  $$x_1^{d_1-1}x_{\alpha_1+1}^{d_2-d_1}x_{\alpha_1+\alpha_2+1}^{d_3-d_2+1}$$
 as  the largest new minimal generator and then
$$ x_1^{d_1-1}x_{\alpha_1+1}^{d_2-d_1}x_{\alpha_1+\alpha_2+1}^{d_3-d_2}x_{\alpha_1+\alpha_2+2},
 \ldots, x_1^{d_1-1}x_{\alpha_1+1}^{d_2-d_1}x_{\alpha_1+\alpha_2+1}^{d_3-d_2}x_{\alpha_1+\alpha_2+\alpha_3}$$ 
 as the other new minimal generators.
 
 
\medskip

\begin{proof}  We need to prove that for all $N \geq \beta_h$, the universal lexicographic ideal $\fL_{\Gamma,N}$ is the same 
as the ideal generated by $\fB_\Gamma$.  
We shall prove all of the following statements simultaneously by induction on  $j$.
\begin{enumerate}
\item Every monomial occurring in the set of generators of $\fL_{\Gamma,N}$ for $N \gg 0$ involves exactly one new variable, 
the next largest variable not already used.
\item All monomials occurring among the generators of $\fL_{\Gamma,N}$ for $N \gg 0$ in degree at least $d_j$ are divisible by $\mu_j$ and, if a variable occurs in $\mu_j$, it cannot occur in a monomial generator of total
degree $> d_j$ with an exponent strictly larger than it has in $\mu_j$.
\item  The minimal generators of $\fL_{\Gamma,N}$ for $N \gg 0$ occurring in degree $d_j$ are precisely
the monomials displayed in $(\dagger_j)$,
and the new variables,
$$x_{\beta_{j-1}+2},  \ldots,  x_{\beta_{j-1}+i}, \ldots,  x_{\beta_{j-1}+\alpha_j},$$ 
 introduced in the monomials after the first
do not occur in any other minimal generator of
$\fL_{\Gamma,N}$ for $N \gg 0$.
 \end{enumerate}

From these statements, we will be able to show easily that $\fB_{\Gamma}$ generates $\fL_{\Gamma, N}$ for all $N  \gg 0$.   
It is clear that the specified generators are  correct for degree $d_1$, since the monomials specified are the largest of 
that degree no matter how many variables $N \gg 0$ there are.  
 Assume the statements above and the specification of generators for $\fL_{\Gamma,N}$, where $N \gg 0$, are correct through degree $d_{j-1}$ for some $j$ with $2 \leq j \leq h$.  We need to show the same for the monomials in degree $d_j$. 
No monomial in degree higher than $d_{j-1}$  can occur with a bigger exponent on any of the variables in $\mu_{j-1}$ 
than it has in $\mu_{j-1}$: if we increase the exponent on $x_{\beta_i +1}$ we get a multiple of the largest minimal generator of degree $d_i$.  
On the other hand,  all monomials chosen in degree $d_j$ or  higher degree must be multiples of $\mu_{j-1}$. The reason is that for sufficiently large $N$, there are arbitrarily many multiples of $\mu_{j-1}$ in every degree:  anything not a multiple of $\mu_{j-1}$ is smaller than these, or else a multiple of a generator of lower degree.     No previously used variable can occur 
in the terms  of degree $d_j$ or higher: if that happened, it would produce a multiple of 
a chosen monomial of lower degree.  The smallest monomial we might use comes from multiplying by a power of the next new variable.  We cannot use that new variable again.  \medskip 

We now discuss the final statements of the theorem. We comment only on the formula $(\dagger)$.  
We have mutually disjoint contributions to the Hilbert function:  since $\mu_j$ has degree $d_j-1$, 
the contribution in degree $t$ from multiples of elements of degree $d_j$ not already in the ideal is the number of
monomials of degree $t - (d_j - 1)$  in the ideal 
$$(x_{\beta_{j-1}+1}, x_{\beta_{j-1}+2},  \ldots, x_{\beta_{j-1}+i}, \ldots, x_{\beta_{j-1}+\alpha_j})S_j.$$
Now apply Lemma~\ref{Mm} with $M = N-\beta_{j-1}$, $m = \alpha_j$, and $t$ replaced by $t-(d_j-1)$. Note that $M-m = N-\beta_j$.
\end{proof}


Recall the definition of the $a_i$ and $b_j$ from Notation~\ref{abnot}.
Note that the formula $(\dagger)$ in Theorem~\ref{stablexdescrip}  Hilbert function is still correct if we insert additional degrees
for which the value of the corresponding $\alpha_j$ is 0: each of the extra summands in the formula is the difference of
two terms that are equal, since $\beta_j = \beta_{j-1}$.  The  terms of the original formula are recovered from the
terms in the new formula where $b_j \not= b_{j-1}$.

\begin{corollary}\label{stlexhilb}  Given a universal lex ideal $\fL$ generated in degree at most $d$, with the $b_j$ as above (i.e., $b_j$ is the 
total number of minimal generators of $\fL$ of degree at most $j$), we have that, for $N > b_d$, the Hilbert function of $\fL$ is
$${}\,\,\,\,(\dagger\dagger) \qquad
t \mapsto \sum_{j=1}^d \bigg(\bco{N-b_{j-1}+ t-j}{N-b_{j-1}-1} - \bco{N-b_j+t-j}{N-b_j-1}\biggr) \quad\hbox{\rm or}$$
$${}\!\!\!\!(\dagger\dagger') \qquad t \mapsto \sum_{j=1}^d \bigg(\bco{N-b_{j-1}+ t-j}{t-j+1} - \bco{N-b_j+t-j}{t-j+1}\biggr).$$
\end{corollary}

\section{The case of Hilbert functions of arbitrary ideals and modules}\label{arb}  

\begin{discussion}\label{extended}  We fix a field $K$.
Suppose that we are given a homogeneous ideal $I$ in $\Rg h$ or a finitely generated $\Z$-graded module $M$ 
over $\Rg h$. Let $H$ denote the Hilbert function of $I$ or $M$.  The function $H$ 
uniquely determines the Hilbert function when 
ideal $I$  or $M$ is extended to $K[\vect x N]$ for any $N \geq h$  (one forms the extension by 
tensoring either with $\Rg N$ over $\Rg h$ or,
equivalently,  with $K[x_{h+1}, \, \ldots, \, x_N]$ over $K$):
upon tensoring over $K$ with a polynomial ring in $N-h$ variables we get a Hilbert function $H^{(N)}$ whose value
on  $t$ is given by $$H^{(N)}(t) = \sum_{0 \leq i \leq t} H(i)\bco{N-h+t-i-1}{t-i}.$$
The second factor in each of the summands is the number of monomials of degree $t-i$ in $N-h$ variables.

As usual, we define twists $M(s)$ for $\Z$-graded modules $M$, where $s \in \Z$, by $[M(s)]_t := [M]_{t+s}$.
Using a graded free resolution by finitely generated free $\RN$-modules and degree preserving maps
(each free module is a finite direct sum of twists $\RN(s)$,  $s \in \Z$ of $\RN$), we may express the 
Hilbert function $\HN$  as a $\Z$-linear combination of binomial coefficient {\it functions} 
$\bco{N+t+s-1}{t+s} = \bco{N+t+s-1}{N-1}$, and so the the Hilbert function $H^{(N)}$ of $\Rg N \otimes_{\Rg h}M$ described 
in the preceding paragraph can be written
in the form  $$\sum_{s \in \Z} c_s\bco{N+t+s-1}{t+s} = \sum_{s \in \Z} c_s \bco{N+t+s-1}{N-1}$$ where the $c_s \in \Z$ and all but finitely many of the $c_s$ are  0.   The $c_s$ are uniquely determined by $I$ or $M$:  if $G_\bullet$ is a finite free graded resolution of $I$ or $M$ over 
$\Rg h$, then $c_s$ is the difference between the total number of occurrences
of $\Rg h(s)$ as a summand of $G_j$ with $j$ even and the total number of occurrences of $\Rg h(s)$
as a summand of $G_j$  with $j$ odd, and is independent of the choice of the graded resolution $G_\bullet$.  
We refer to this function of $N$ and $t$, which is defined for all $t \in \Z$ and all $N \geq h$, as the {\it extended Hilbert
function} of $I$ or $M$.   

We shall say that a function of $N \geq 1, t \in \Z$ is an {\it LCBC function} (LCBC stands for ``linear combination of binomial coefficients") if it is a
$\Z$-linear combination of the functions $\bco{N+t+s-1}{t+s} = \bco{N+t+s-1}{N-1}$, where $s$ varies in $\Z$. We shall say that a polynomial in 
$\Q[Z,t]$ is an {\it LCBC polynomial} if it a $\Z$-linear combination of the  corresponding functions $\bcp{N+t+s-1}{N-1}$, which are polynomial
in $t$.  Thus, $s$ indexes a $\Z$-basis for the LCBC functions (respectively, polynomials). \end{discussion}

\begin{remark}\label{convert}  Given an LCBC function $G$ with coefficients $c_s$,  if we fix  $t = d$, it agrees with 
a polynomial $G_d$ in $N$ for $N \gg 0$.  
To get  $G_d(N)$ when  $G$ has coefficients $c_s$,  in  $\sum_{s \in \Z} c_s\bco{N+d+s-1}{N-1}$  we may replace  $c_s\bco{N+d+s-1}{N-1}$  by $c_s\bcp{N+d+s-1}{N-1}$
if $s \geq -d$.  However, if $s < -d$,  the term $c_s\bco{N+d+s-1}{N-1}$ vanishes for all $N \gg 0$,  and must be replaced by 0. 
If $s \geq -d$,  $$c_s\bcp{N+d+s-1}{N-1} = c_s\bcp{N+d+s-1}{d+s},$$ 
and the constant term of this polynomial in $N$  is 0 unless $s = -d$,  in which case it is $c_{-d}$.  Hence:
   \end{remark}

\begin{proposition}\label{findG} Let $G$ be an extended Hilbert function as in Discussion~\ref{extended} 
such that $$G(N,t) = \sum_{s \in \Z} c_s\bco{N+t+s-1}{N-1}.$$
If we fix $t =d$, the polynomial $G_d(N)$ that agrees with $G(N,d)$ for all $N \gg 0$ is $$\displaystyle \sum_{s \geq -d} c_s\bcp{N+d+s-1}{d+s},$$ and the constant term $G_d(0)$ of $G_d(N)$ is $c_{-d}$.  \qed \end{proposition}
      
\begin{theorem}\label{LCBC}  The coefficients $c_s$ that occur in the description of an LCBC  function are uniquely determined 
by the values of the
function for $N \gg 0$ and $t \gg 0$.  Hence, with each LCBC function there is a uniquely associated LCBC polynomial.

Moreover, if two LCBC functions agree for all $N \gg 0$ in a fixed degree $t = d$, then their difference is constant for all $N \gg 0$
when $t = d+1$.  In  fact, suppose one function $Q$ has coefficients $c_s$ and the other $Q'$ has coefficients $c'_s$.  Then they agree
for all $N \gg 0$ in  degree $d$ if and only if $c_s = c'_s$ for $s \geq -d$, in which case $Q(N, d+1) - Q'(N,d+1) = c_{-d-1} - c'_{-d-1}$
for all $N \gg 0$.  
\end{theorem}

\begin{proof}  The  result of the second paragraph implies the result of the first.  We may work with $Q-Q'$ and 0.  For  $s < -d$
the function with $c_s-c'_s$ as coefficient in $Q-Q'$ has a negative denominator and vanishes.  
As $s$ takes on the values $-d, -d+1, -d+2, \ldots, -d+j, \ldots $
$Q-Q'$ cannot vanish identically for large $N$ unless all the coefficients $c_s - c'_s$ are 0.  But then, when we substitute 
$d+1$ for $d$, the only term that does not vanish is the one where $s = -d-1$, and one gets the specified constant
$c_{-(d+1)} - c'_{-(d+1)}$ as the value. \end{proof}

\begin{remark}\label{sub} The function  $\bcp{N +\gamma}{N+\delta}$ with $\gamma \geq \delta$ (where $\gamma$, $\delta$ may be
negative) equals $\bcp{N + \gamma}{\gamma - \delta}$, and the constant term is $\bcp{\gamma}{\gamma-\delta}$.  \end{remark}

\begin{discussion}\label{mainrec} For every $N \geq h$, there is a unique lex ideal $\cL_{I,N}$ in $\Rg N$ that has
the same Hilbert function as $I\Rg N$: see \cite{Mac} or \cite{Mer} for an expository version. 
We shall prove that for a fixed degree $d$, these lex ideals eventually 
all have the same set of minimal generators 
in degree at most $d$, and are  universal lex.  We shall let $\fB_{I,d}$ denote the set of generators in degree at most $d$.   
We prove this by induction on $d$. 
Suppose that we have constructed the ideal up to degree $d$,  and we want to construct it
in degree $d+1$. We have a formula for the Hilbert function of $I\Rg N$ and a formula for the Hilbert function
of the approximation up through degree $d$.  As functions of $N, t$ these  agree in degree $d$ for all $N \gg 0$.
By the result above, the difference of their values in degree $d+1$ is constant for all $N \gg 0$.  Denote this
constant $a_{d+1}$.  If we take a new universal lex ideal with $a_{d+1}$ new minimal generators in degree $d+1$,
it will agree with the Hilbert function of $I\Rg N$ in all degrees $\leq d+1$ for all $N \gg 0$.  This will lead
to a formula for $a_{d+1}$ in terms of $\vect bd$ and the coefficients in the LCBC that gives the Hilbert function of $I$.
Then $b_{d+1} = b_d + a_{d+1}$.  \end{discussion}

We carry this out in detail.   Suppose that $I$ is a homogeneous ideal of $\Rg h$ whose extended Hilbert function has
integer coefficients  $c_s(I)$.  Here, $s$ varies in $\Z$ but has only finitely many nonzero values.   We develop recursive
formulas for $b_d(I)$, which is the least number of generators of the universal lexicographic ideal whose Hilbert function
agrees with that of $I$ in degrees up to and including $d$.  Thus $b_d = 0$ if $d$ is strictly less than the degree of any
minimal generator of $I$, and if $d_0$ is the least degree of a minimal generator of $I$,  $b_{d_0}$ is the number of
minimal generators of $I$ of degree $d_0$.

We write  $a_d(I)  = b_d(I) - b_{d-1}(I)$.   For the rest of this discussion, we omit $I$ from the notation, and abbreviate
$b_d = b_d(I)$ and $a_d = a_d(I)$. 
Suppose that we know  $b_j$ for $j \leq d$.  We obtain the recursion as follows.  
First, we have that $a_{d+1}$ is the difference between the value of the extended Hilbert function of $I$ in degree $d+1$ and the extended
Hilbert function of the universal lex ideal we have already constructed with generators in degree at most $d$:  moreover, we know that
the difference is constant, independent of $N$, by Theorem~\ref{LCBC}.  Second, we have a formula $(\dagger\dagger')$ for the extended
Hilbert function of the universal lex ideal for $N \gg 0$ by  Corollary~\ref{stlexhilb}.  Hence, by Remark~\ref{convert} and Proposition~\ref{findG},
$$\quad a_{d+1} = c_{-(d+1)} - \sum_{j=1}^d \bigg(\bcp{N-b_{j-1}+(d+1)-j}{d-j+2} - \bcp{N-b_j+(d+1)-j}{d-j+2}\biggr),$$
and the value of this expression is independent of $N$, since it is a polynomial in $N$ and constant for $N \gg 0$.  

By Remark~\ref{sub}, we may substitute $N = 0$ to get the constant value
$$ a_{d+1} = c_{-(d+1)}- \sum_{j=1}^d \bigg(\bcp{-b_{j-1}+d-j+1}{d-j+2} - \bcp{-b_j+d-j+1}{d-j+2}\biggr).$$
 
For $1 \leq j \leq d-1$,  there are two terms in the summation involving $b_j$:  one of these is the second term occurring for index $j$ and the other
is the first term occurring for index $j+1$.  These two  terms may be combined using Equation \ref{bc2}:
$$(\#)\quad - \bcp{-b_j + d-j+1}{d-j+2} + \bcp{-b_j+d-j}{d-j+1} = -\bcp{-b_j + d -j}{d-j+2}.$$ 

Using Equation \ref{bc1}, this becomes
$$-(-1)^{d-j}\bcp{b_j+1}{d-j+2}.$$
 
This yields 
$$a_{d+1} = c_{-{d+1}} - \Biggl(\bcp{-b_0 + d}{d+2} -\biggl(\,\sum_{j=1}^{d-1}(-1)^{d-j} \bcp{b_j+1}{d-j+2}\!\biggr) -
\bcp{-b_d+1}{2}\Biggr).$$  By Equation \ref{bc1}, $\dsp \bcp{-b_d+1}{2} = \bcp{b_d}{2}$.  Assuming that $I$
is not the unit ideal, we have that $b_0 = 0$, and so  $\bcp{-b_0 + d}{d+2}  = 0$.  Hence:
$$a_{d+1} = c_{-{d+1}} +  \Bigl(\bcp{b_d}{2} + \sum_{j=1}^{d-1}(-1)^{d-j} \bcp{b_j+1}{d-j+2}\Bigr).$$

By adding $b_d$,  we obtain a formula for $b_{d+1}$.  

\begin{theorem}\label{main} Let $I$ be a proper homogeneous ideal in $\Rg h$ and abbreviate $c_s = c_s(I)$, which is 
defined in Discussion~\ref{extended}, and $b_j = b_j(I)$. Then:
$$b_{d+1}  = c_{-(d+1)}  + b_d +\bcp{b_d}{2} + \sum_{j=1}^{d-1}(-1)^{d-j} \bcp{b_j+1}{d-j} =
c_{-(d+1)} + \sum_{j=1}^d(-1)^{d-j} \bcp{b_j+1}{d-j} $$
\qed
\end{theorem}
\bigskip
 
\section{Hamilton numbers}\label{Ham} 

In this section we describe a sequence of integers studied by Hamilton and
described in \cite{Luc}, \cite{LucSl}, \cite{OEIS} and \cite{SylH}  that originally arose in studying 
the behavior of generalizations of the Tschirnhaus transformation.\footnote{Briefly, $H_n$ is the least degree of an equation from which $n$
consecutive terms after the term of highest degree can be eliminated by a sequence of transformations related to Tschirnhaus
transformations without needing to solve an equation of degree $>n$ to determine the transformation.}   
For further background we refer the
reader to \cite{CHM} and \cite{Gar} as well as to the papers already cited.    Of course, the Hamilton numbers arise here in
a completely different context. 

\newcommand\mc[1]{\multicolumn{1}{|r}{#1}}

We first give a self-contained treatment of the Hamilton numbers that includes all of the results we need.  We then
connect this treatment with the one given in \cite{Luc}, which has two small errors in it, one of which is noted in 
\cite{LucSl}. 

We begin by defining a sequence $\ell_n$ recursively for $n \geq 0$ by letting $\ell_0 = 3$
 and using the recursion 
  \begin{equation}\label{ell}
 \ell_{n+1} = 1 + \bco{\ell_{n}} 2 - \bco{\ell_{n-1}} 3 + \cdots + (-1)^n\bco{\ell_0} {n+2}, \hbox{\ } n \geq 0.
  \end{equation}

 The general term on the right, after the initial 1,   is  $(-1)^{j}\bco{\ell_{n-j}} {j+2}$, where $0 \leq j \leq n$.   
 Note that, typically, many of the rightmost terms vanish.  We shall soon see that $\ell_n >0$ for all $n \geq 0$, so that 
 the formula just above is also correct if we replace
 $\bco {\ell_{n-j}} {j+2}$ by $\bcp {\ell_{n-j}} {j+2}$, since these agree when the numerator is positive. 
 
We may then define $H_n$ for $n \geq 1$ by the formula $H_n:= \ell_{n-1}-1$.  We have: \bigskip

\centerline
{$\begin{array}{|r|r|r|} 
                                   \hline
                                  n & \ell_n & H_n \\
                                  \hline
                                 0  & 3 &{} \\
                                 \hline
                                 1 & 4 & 2\\
                                 \hline
                                 2 & 6 & 3 \\
                                 \hline
                                 3& 12 & 5\\
                                 \hline
                                 4& 48 & 11\\
                                 \hline
                                 5 & 924 & 47\\
                                 \hline
                                 6 & 409620 & 923\\
                                 \hline
                                 7 &    83763206256 & 409619\\
                                 \hline
                                 8 &  3508125906290858798172  & 83763206255\\
                                 \hline
                                 9 & 6153473687096578758448522809275077520433168 &3508125906290858798171\\
                                 \hline
                                 \end{array}$ }  \bigskip

 For the curious, $\ell_{10}$ is 189326192088949818333335820590333293708012662495359 
02023330546944758507753065602135844 which is larger than the approximation $1.89 \cdot10^{85}$. 
\medskip
 
 The growth rate of $\ell_n$ is double exponential.  More precisely:
 
 \begin{proposition}\label{Hamprop} We have the following:
 \begin{enumerate}[(a)] 
 
 \item For all $n \in \N$,  
 $\dsp { \ell_n^2 \over 3}\leq \ell_{n+1} \leq {\ell_n^2 \over 2}$.  
 
 \item If $M > L > k \geq 2$ are integers such that  $\dsp M \geq {1 \over 3}L^2$  then $\dsp \bco M k \geq \bco L {k+1}$. 
 
 \item  For all $n \in \N$, the binomial coefficient terms $\dsp \bco {\ell_{n-j}} {j+2}$
on the right hand side of the recursion Eq.~\ref{ell} for $\ell_{n+1}$ are nonincreasing,  
$0 \leq j \leq n$.  \smallskip 
 
 \item
$\dsp \lim_{n \to \infty} {\ell_{n+1} \over \ell_n^2} = {1 \over 2}.$ \smallskip  

\item For all $n \geq 1$, $3\cdot2^{2^{n-1}} \leq \ell_{n+1}$, and for all $n \geq 0$, $\ell_{n+1} \leq  2\cdot 2^{2^n}$. 

\item $\ell_n$ is asymptotic to $2 \cdot 2^{\rho 2^n}$ where $$\rho = 0.2756687129668628532825852274380553674012976$$ 
 to 43 decimal places.  If we let $\dsp \rho_n : ={\log_2(\ell_{n+1}) -1 \over 2^n},$ then for all $n\geq  4$,  $\rho_n \geq \rho \geq
 \rho_n -{1\over  2^{n-3}\ln 2\sqrt{\ell_n}}$.    

\end{enumerate}
 \end{proposition} 
 \begin{proof} We first prove (b).  Since $(k+1)!/k! = k+1$ and since, for $0 \leq j < L$,
 $\dsp  {M - j \over L-j} \geq {M \over L} $, we have\hfill\smallskip
 
\noindent $\dsp {\bco M k \over \bco L {k+1}}
 = (k+1) {M \over L}{M-1 \over L-1} \cdots {M-j \over L-j} \cdots {M-(k-1) \over L-(k-1)} {1 \over L-k} \geq (k+1)\Bigl({M\over L}\Bigr)^k{1 \over L} \geq$\hfill \break
$\dsp 3\Bigl({M \over L}\Bigr)^2 {1 \over L}
= {3M^2 \over L^3} \geq  {3 \Bigl({L^2 \over 3}\Bigr)^2 \over L^3}= {L \over 3} \geq 1.$ \medskip

We next prove (a) and (c) simultaneously by induction on $n$.  
One may check (a) and (c)  explicitly for $0 \leq n \leq 4$ by calculating $\ell_n$ for $0 \leq n \leq 5$.
Assume both (a) and (c) hold for integers up to and including $n$.  In the recursion for $\ell_{n+1}$,
the binomial coefficients on the right are nonincreasing because the first of the inequalities in part (a) holds for 
$M = \ell_{n-j}$ and $L= \ell_{n-j-1}$  until we reach the terms that are 0.  The alternating sum of the terms 
of a nonincreasing sequence of nonnegative real numbers is nonnegative, from which we have at once
that  
$$(\dagger)\quad   1 + \bco {\ell_n}  2- \bco {\ell_{n-1}} 3  \leq \ell_{n+1} \leq 1 + \bco {\ell_n}  2.$$

To complete the inductive proof of (a) and (c), we abbreviate $L := \ell_n$.  The second inequality
in $(\dagger)$  yields  $\dsp \ell_{n+1} \leq 1 +{ L(L-1) \over 2} \leq {L^2 \over 2}$.
The first inequality in $(\dagger)$ together with the fact that $\dsp L \geq {\ell_{n-1}^2 \over 3}$,
so that $\ell_{n-1} \leq (3L)^{1/2}$  and  $\bco {\ell_{n-1}} 3 \leq {(3L)^{3/2} \over 6}$ yields
 $$(*) \quad  \ell_{n+1} \geq 1 + {L(L-1) \over 2} -  {3\sqrt{3}L\sqrt{L}\over 6},$$
 and so we are done if $\dsp {L^2 \over 3} \leq {L^2 -L \over 2} -{\sqrt{3}L\sqrt{L} \over 2}$.
 If we multiply by $\dsp {6\over L\sqrt{L}}$ and collect terms, we see that this is equivalent to
 $\dsp \sqrt{L} \geq {3 \over \sqrt{L}} + 3\sqrt{3}$.  This is true if $L \geq 48$, and 
 the cases that involve a smaller value of $L$, where $0 \leq n \leq 4$, are handled by explicit calculation in the
 base case for the induction.  Hence, (a) and (c) are proved.
 
 With $L = \ell_n$,  the inequalities
 $$1 + {L(L-1) \over 2} - {3\sqrt{3} L\sqrt{L} \over 6}   \leq \ell_{n+1} \leq  1 + \bco L 2$$
 were established in $(*)$ and $(\dagger)$ in the preceding paragraph.  If we divide by $L^2$, we
 see that $\dsp {\ell_{n+1}\over \ell_n^2}$ is trapped between two numbers both of
 which approach ${1 \over 2}$ as $n \to \infty$.   This establishes part (d).
 
 Part (e) now follows from the inequalities in part (a) by
 a straightforward induction.
 
For part (f), for $n \geq 0$ let $\dsp \rho_n : ={\log_2(\ell_{n+1}) -1 \over 2^n},$  so that  
$\ell_{n+1} = 2 \cdot 2^{\rho_n  2^n}$.  Note that $\rho_0 = 1$,  that every
$\rho_n \geq 0$, and that $\rho_{n+1} < \rho_n$ since 
$$\ell_{n+2}  < {1 \over 2}\ell_{n+1}^2 \imp  \log_2(\ell_{n+2}) < 2\log_2(\ell_{n+1}) -1.$$  
Hence, as $n \to \infty$, $\rho_n$ converges to $\rho \geq 0$.
Also, by $(*)$ above with $n$ increased by one, 
$$\ell_{n+2} \geq 1 + \bco {\ell_{n+1}} 2 - {{\sqrt 3} \over 2}\ell_{n+1}\sqrt{\ell_{n+1}} = 
1 + {\ell_{n+1} \over 2}(\ell_{n+1} -1 - \sqrt 3  \sqrt{\ell_{n+1}})$$   
and since $\sqrt 3 \sqrt{\ell_{n+1}} \leq 2 \sqrt{\ell_{n+1}}$  for all $n$, we have
$$ \ell_{n+2} \geq 1 + {\ell_{n+1}^2 \over 2} - {\ell_{n+1} \over 2} - \ell_{n+1}^{3/2} \geq  {\ell_{n+1}^2 \over 2}  - 2\ell_{n+1}^{3/2},  \,\quad n \geq 3.$$

Hence:
$$ (\#) \quad 2 \cdot 2^{\rho_{n+1} 2^{n+1} } \geq 
 {1 \over 2}\bigr(2\cdot 2^{\rho_n  2^n}\bigr)^2 - 2(2 \cdot 2^{\rho_n 2^n})^{3\over 2} = $$
$$ 2^{\rho_n 2^{n+1}+1} - 2^{3\rho_n 2^{n-1} + 5/2}  =  2^{\rho_n 2^{n+1}+1}(1 - {1\over u_n})$$
 where $$ (\ddagger) \qquad u_n = {2^{\rho_n 2^{n+1}+1} \over  2^{3\rho_n 2^{n-1} + 5/2}} = 2^{\rho_n 2^{n-1} - 3/2} = {1 \over 4}\sqrt{\ell_{n+1}}.$$
  Note that if $n \geq 4$, $u_n > 2$ and $u_n$ increases with $n$.  
  Applying $\log_2\ $ to the first and last terms in $(\#)$, we have
 $$1 + \rho_{n+1}2^{n+1} > \rho_n  2^{n+1} +1 + \log_2(1 - {1\over u_n}) \hbox{\ or\ }
 \rho_{n+1} > \rho_n +{1 \over 2^{n+1}} \log_2(1 - {1 \over u_n}).$$
  Note that for $\dsp\log_2(1 -x) = {\ln (1-x) \over \ln 2} = {-1 \over \ln 2}(x + x^2/2 + \cdots)$ is negative
  for $0 < x < 1$,  and $\dsp x + x^2/2 + x^3/3 + \cdots  <  x + x^2 + x^3 + \cdots = {x \over1-x} < 2x$ if $0 < x < 1/2$.  
  Hence, once $n \geq 4$, so that $u_n > 2$,   $$\rho_{n+1} > \rho_n - {1 \over 2^{n+1}}{1 \over \ln 2} {2 \over u_n}.$$
    This shows that $\rho_{n+1}$ is trapped between
   $\rho_n$ and $\rho_n - \epsilon_n$ with $\dsp\epsilon_n = {1 \over (\ln 2) 2^n u_n}.$
   Clearly, $\dsp \epsilon_{n+1} < {\epsilon_n \over 2}$, and it follows that 
   $$0 < \rho_n - \rho < (1 + {1\over 2} + {1\over 4} + \cdots )\epsilon_n = 2\epsilon_n = {1\over 2^{n-3} (\ln 2) \sqrt{\ell_{n+1}}},$$
   using the formula $(\ddagger)$ for $u_n$ displayed above. Thus, $\rho_n$
   converges to $\rho$ very rapidly.  Now,
   $${\ell_{n+1} \over 2\cdot 2^{\rho 2^n}} = {2\cdot 2^{\rho_n2^n} \over 2\cdot 2^{\rho 2^n}} 
   = 2^{(\rho_n - \rho)2^n} < 2^{2\epsilon_n 2^n}
   = 2^{{2\over (\ln 2)u_n}} \to 1$$ as $n \to \infty$,
   which establishes the asymptotic estimate in part (f).  
   
To estimate $\rho$ using $\rho_9$,  note that since  $\ell_{10} > 1.89 \cdot 10^{85}$, we have that  
$$\rho_9 - \rho < 2 \epsilon_9 <  {1 \over 64 (\ln 2)\sqrt{18.9} \cdot 10^{42}} < 5.19\cdot10^{-45}.$$  
Since the 44 th and 45 th digits of $\rho_9$ after the decimal point are both 3, this error will not affect earlier digits, and
$\rho_9$ will agree with $\rho$ to at least 43 decimal places, which is what is shown in part (f). 
 \end{proof}
             
We shall need:
 
\begin{proposition}\label{recurl} The following alternate recursive formulas for $\ell_n$ hold, as well as the recursive formula
shown for $H_{n+1}$.
\begin{enumerate}[(a)]
\item For all $n \geq 1$,
   $\dsp  \quad \ell_{n+1} = \sum_{j = 1}^{n+1} (-1)^{j+1}{ \ell_{n - j+1} + 1\choose j+1}.$    
   
\item  For all $n \geq 0$, $ \dsp \ell_n = 3 + \sum_{j=0}^{n-1}(-1)^j{\ell_{n-j-1}-1 \choose j+2}.$  

\item For all $n \geq 0$, $\dsp H_{n+1} = 2 + \sum_{j=0}^{n-1}(-1)^j{H_{n-j} \choose j+2}.$  
\end{enumerate}
\end{proposition}
\begin{proof} (a) For $n =1$, this may be verified by direct calculation.  Assume $n \geq 2$.
Let E$_n$ be the equation displayed as Eq.~(4.1).  
$$ (\hbox{E}_n)\qquad  \ell_{n+1} = 1 + \bco{\ell_{n}} 2 - \bco{\ell_{n-1}} 3 + \cdots + (-1)^n\bco{\ell_0} {n+2}, \hbox{\ } n \geq 0.$$

Since $n \geq 2$ we may 
write E$_{n-1}$ as a formula for $\dsp \bco {\ell_n} 1$, subtract the equation E$_{n-1}$ from
E$_n$,  transpose $\dsp \bco {\ell_{n-1}} 1$ to the right hand side of the equation,
and apply Equation~\ref{bc3}(i)  to all pairs of terms
with the same numerator and consecutive integer denominators to obtain the required
result. 

The recurrence in (b)  follows by induction on $n$. For the inductive step, if $n \geq 1$, note that
E$_{n-1}$ can be rewritten as
$$\ell_n  =  \bco{\ell_{n-1}-1} 2  +\ell_{n-1} + \sum_{j=1}^{n-1}(-1)^{j}\bco{\ell_{n-1-j}} {j+2}.$$
by substituting $\dsp 1 +  {\ell_{n-1} \choose 2} =  {\ell_{n-1}-1 \choose 2} +\ell_{n-1}$ on the 
right side.
We replace $\ell_{n-1}$, the second term on the right,  by the formula from the equation obtained by replacing $n$ 
by $n-1$ in (b),  which is the induction hypothesis.  This yields
$$\ell_n = \bco{\ell_{n-1}-1} 2  + \Biggl(3 + \sum_{j=0}^{n-2}(-1)^j{\ell_{n-2-j}-1 \choose j+2}\Biggr) +
\sum_{j=1}^{n-1}(-1)^{j}\bco{\ell_{n-1-j}}{j+2}.$$
We shift the summation variable in the first sum by 1 so that it varies from 1 to $n-1$ and combine terms to obtain:
$$\ell_n = 3 +\bco{\ell_{n-1}-1} 2 + \sum_{j=1}^{n-1}\biggl((-1)^{j-1}{\ell_{n-1-j}-1 \choose j+1}+(-1)^{j}\bco{\ell_{n-1-j}}{j+2}\biggr).$$ 
The required formula for $\ell_n$ is now obtained by applying Equation~\ref{bc3}(ii) to the terms in the summation.

One obtains (c) from (b) by replacing each occurrence of $\ell_k$  by $H_{k+1} + 1$  and subtracting 1 from both sides of the equation. 
\end{proof}

We next give a description of the Hamilton numbers following \cite{Luc}.  We generate an array, shown below,
in which the zeroth row consists of $1$ followed by an infinite sequence of zeros. We use the notation 
$\lambda_{ij}$, $i,j  \geq 0$ for the $j\,$th entry of the $i\,$th row.  The rows have indentations from the left that are nondecreasing.  The indentation 
$$\begin{array}{rrrrrrrrrc}  
\hline
\mc{\mathbf{1}} & 0  & 0 & 0 & 0  & 0    & 0   & 0 &\cdots \\ 
                                    \cline{1-9}                 
       &\mc{\mathbf{1}} & 1 & 1 & 1  & 1    & 1 & 1  &\cdots \\
                             \cline {2 - 9}                        
                      {}  & {} &\mc{\mathbf{2}}& 3 & 4  & 5    & 6 & 7 &\cdots \\
                                {}  & {} &\mc{1} & 5 & 9  & 14 & 20 &27 &\cdots \\
                       \cline {3-9}
                                         {}  & {} & {} &\mc{\mathbf{6}} & 15 & 29 & 49 &76 &\cdots \\
                                          {}  & {} & {} &\mc{5} & 21 & 50 & 99 &175&\cdots \\
                                          {}  & {} & {} &\mc{4} & 26 &76  & 175 &350&\cdots \\   
                                       {}  & {} & {} &\mc{3} & 30 &106 & 281 & 631&\cdots \\  
                                    {}  & {} & {} & \mc{2} &33 &139 & 420 & 1051&\cdots \\  
                                          {}  & {} & {} & \mc{1} & 35 &174 & 594 & 1645&\cdots \\ 
                           \cline{4-9} 
                                        {}  & {} & { }& {} & \mc{\mathbf{36}} & 210 & 804 &2449&\cdots \\ 
                                        {}  & {} &  & {} &\mc{35} & 246 & 1050 &3499&\cdots \\ 
                                         {}  & {} & {} & {} &\mc{34} & 281 & 1331 & 4830&\cdots \\ 
                                      {}  & {} & {} & {} & \mc{33} & 315 & 1646 & 6476 &  \cdots \\    
                                             {}  & {} & {} & {} & \mc{32} & 348 & 1994 & 8470 &\cdots \\      
                                             {}  & {} & {} & {} &\mc{\,\,\vdots} & \vdots\,\,\,\, & \vdots \,\,\,\,\, & \vdots\,\,\,\,\,\, &\cdots \\
                                               {}  & {} & {} & {} & \mc{1} & 875 & 23694 & 3765664 &\cdots \\
                                          \cline{5-9}
                                                 {}  & {} & {} & {} &  & \mc{\hbox{\bf{876}}} & 24570& 401134 &\cdots \\ 
                                                {}  & {} & {} & {} &  & \mc{\,\,\,\,\,\vdots} & \vdots\,\,\,\,\, & \vdots\,\,\,\,\,\,&\vdots\,\,\, \\ 
                                \end{array}$$   
of the $i+1\,$th  row is the same as for
the $i\,$th row if the first entry of the $i\,$th, call it $\lambda = \lambda_{i, j_i}$,  is greater than 1.  In this case, the first
entry of the $i+1\,$th row is $\lambda - 1$, and for all other entries, $$(**) \quad \lambda_{i+1,j} = \sum_{t \leq j} \lambda_{i,t}.$$   
If the first entry of the $i\,$th row is 1, the indentation increases by 1, and all entries of the $i+1\,$th row are given by
the formula $(**)$.  We refer to all the consecutive rows with the same indentation as a {\it block}.  The blocks in the
table above are separated by horizontal lines.

Consider the sequence of entries beginning rows where the indentation has just increased by 1,  but including the zeroth row.  These entries are shown in boldface in the table above. This
is the sequence consisting of the first entries of the initial rows of the blocks.  
This sequence is $1, 1, 2, 6, 36, \ldots$.       Call this sequence $a_n$,  $n = 0, 1, 2, 3, 4, \ldots$.     (This use
of the notation $a_n$, which is consistent with \cite{Luc},  occurs in this paper only in this paragraph and the next: 
 it has a different meaning in the later sections.)
Note that for $n \geq 1$ there are $a_n$ rows with indentation $n$,  and that their leading entries are the consecutive integers
from $a_n$ to 1 in descending order.   

Let $s_{n+1} =  a_0 + \cdots + a_n $,  for $n \geq 0$, where $s_0 = 0$.     In \cite{Luc},  $H_n$ is defined as
$s_n+1$  for $n \geq 1$.  Thus  $H_1 = s_1 + 1 =  2$.  We extend the definition so that $H_0 = s_0 + 1 = 1$.  
A recursive formula for $H_{n+1}$ is then derived in \cite{Luc} from properties in the description of the table above:
this formula is displayed at the top of page 498 of \cite{Luc}. 
However, there is an error in this formula, as noted in \cite{LucSl}:  the term subtracted
on the left should have been 2 rather than 1.  The corrected formula leads to a recursion recorded in \cite{OEIS}
(by giving formulas in Maple and Mathematica), namely:
$$H_1 :=2 \hbox{\ \ and\ \ } H_{n+1} = 2 + \sum_{i=0}^{n-1}(-1)^i{H_{n-i} \choose i+2},\ n \geq 1.$$
This agrees with part (c) of Proposition~\ref{recurl}, which shows that our treatment of the Hamilton
numbers yields the same sequence as that in \cite{Luc}.   We note that \cite{Luc} also treats the
numbers $\ell_n$,  but there is another error, because one has $\ell_n := H_{n+1} + 1$, $n \geq 0$,
and in \cite{Luc},  $H_n$ is used instead of $H_{n+1}$.  

Note that our recursion for $\ell_{n+1}$ is the same as the recursion given on page 498 of \cite{Luc}, except
that we have replaced $n$ by $n+1$.

\section{The case of a regular sequence of  two quadratic forms}\label{2quform}  

The next result and the variant that follows describe the behavior obtained from a regular sequence of quadratic forms, making the connection with the
Hamilton numbers noted in \S\ref{intro}.  This result is somewhat surprising, since it means that the number of generators needed for the lex 
ideal whose Hilbert function agrees with the Hilbert function of an ideal generated by a regular sequence of two quadratic
forms has double exponential growth.  \medskip

\begin{theorem}\label{ham}  Over any field $K$,  if $I$ is the ideal generated by $x_1^2, x_2^2$ in  $K[x_1, x_2]$,
 then $b_d(I) = H_{d-2} + 1$ for $d \geq 2$.    
\end{theorem}

This is clearly correct for $d = 2$.  For $d \geq 3$ we may restate this result as:

\begin{theorem}\label{lham} Over any field $K$, with $I$ as above,  $b_d(I) = \ell_{d-3}$ for $d \geq 3$.  \end{theorem}

\begin{proof} The extended Hilbert function for an ideal $I$ generated by a regular sequence consisting of two quadratic forms,
from the Koszul complex resolution $$0 \to {\Rg N}(-4) \to {\Rg N}(-2)^{\oplus 2} \to I \to 0$$ 
is $(N,t) \mapsto 2\bco{N+t-3}{N-1} - \bco{N+t-5}{N-1}$.  That is,  $c_{-2} = 2, c_{-4} = -1$ and all other values of $c_s$ are
0.   We know that $b_0 = b_1 = 0$ and $b_2 = 2$.  The recursion is
$$b_{d+1}  = c_{-(d+1)}  + b_d +\bcp{b_d}{2} + \sum_{j=1}^{d-1}(-1)^{d-j} \bcp{b_j+1}{d-j+2}$$  
$$= c_{-(d+1)}  +\bcp{b_d+1}{2} + \sum_{j=1}^{d-1}(-1)^{d-j} \bcp{b_j+1}{d-j+2} = c_{-(d+1)} + \sum_{j=1}^d (-1)^{d-j} \bcp{b_j+1}{d-j+2} $$ 
Since we have $b_0 = b_1 = 0$ and $b_2=2$ we have the following:  \medskip

\quad When $d = 2$, we get $b_3 = 0 + (-1)\bcp{1}{3}  + \bcp{3}{2}  = 3$.

\quad When $d = 3$, we get $b_4 = -1 + \bcp14 - \bcp33 + \bcp42 = -1 + 0 - 1 + 6 = 4$.  \medskip

To prove that  $b_d = \ell_{d-3}$ for $d \geq 3$  (this is checked above for $d = 3,\, d=4$),  it will be convenient to
extend the definition of $\ell$ so that $\ell_{-3} = 0$, $\ell_{-2} = 0$, and $\ell_{-1} = 2$.   Then we want to show by
induction that $b_d = \ell_{d-3}$ for all $d \geq 0$.  From this definition and the calculations above  we have that
$b_d = \ell_{d-3}$ for $0 \leq d \leq 4$.    For $d \geq 4$ we have that $c_{-(d+1)} = 0$ and so the recursion for
$b_{d+1}$ becomes
$$b_{d+1} = \sum_{j=1}^d (-1)^{d-j} \bcp{b_j+1}{d-j+2}, \ \ d\geq4,$$
and the induction hypothesis yields that
$$b_{d+1} = \sum_{j=1}^d (-1)^{d-j} \bcp{\ell_{j-3}+1}{d-j+2}, \ \ d\geq4,$$
To complete the proof, it suffices to show that the right hand side is $\ell_{d+1-3} = \ell_{d-2}$.
By Proposition~\ref{recurl}(a) we have that for $n \geq 1$, 
 $$\dsp  \quad \ell_{n+1} = \sum_{i = 1}^{n+1} (-1)^{i+1}{ \ell_{n - i+1} + 1\choose i+1}.$$
 The required result follows by letting $n = d-3$ and $j = d+1-i$ (note the last sentence of the first
 paragraph of Notation~\ref{bcop}).
 \end{proof}

\quad\bigskip

$\begin{array}{ll}
\textrm{Department of Mathematics}           &\qquad\qquad \textrm{Altair Engineering}\\
\textrm{University of Michigan}                    &\qquad\qquad \textrm{1820 E.\ Big Beaver Rd.}\\
\textrm{Ann Arbor, MI 48109--1043}            &\qquad\qquad \textrm{Troy, MI 48083}\\
\textrm{USA}                                                &\qquad\qquad \textrm{USA}\\
\quad & \quad\\
\textrm{E-mail: hochster@umich.edu}          &\qquad\qquad \textrm{E-mail: antigran@gmail.com}\\ 
     
\end{array}$


\bigskip

 
\begin{thebibliography}{99}

\bibitem{AH1} Tigran Ananyan and Melvin Hochster, \emph{Ideals generated by quadratic polynomials}, Math. Research Letters {\bf  19} (2012), pp.~233--244.
\bibitem{AH2} Tigran~Ananyan and Melvin~Hochster, \emph{Small subalgebras of polynomial rings and Stillman's conjecture}, preprint, arXiv:1610.09268v2 [math.AC] , 
to appear in the Journal of the Amer.~Math.~Soc.
\bibitem{AH3} Tigran~Ananyan and Melvin~Hochster, \emph{Strength conditions, small subalgebras, and 
Stillman bounds in degree $\leq 4$}, preprint, arXiv:1810.00413 [math.AC].
\bibitem{Big} Anna Bigatti,\emph{Upper bounds for the  Betti numbers of a given Hilbert function}, Comm.~Alg.~{\bf 21} (1993) 2317--2334. 
\bibitem{CCMPV} Giulio Caviglia, Marc Chardin, Jason McCullough, Irena Peeva, Matteo Varbaro \emph{Regularity of prime ideals}, Math. Z. {\bf 291} (2019), no. 1-2, 421--435. 
\bibitem{CHM} Alexander Chen, Yang-Hui He, John McKay, \emph{Erland Samuel Bring's ``Transformation of Algebraic Equations."}, preprint, arXiv:1711.09253 [math.HO], 2017. 
\bibitem{Dr} Jan Draisma, \emph{Topological Noetherianity of polynomial functors}, J.~Amer.~Math.~Soc. {\bf 32} (2019)  691--707. 
\bibitem{DrLL} Jan Draisma, Michal Laso\'n, Anton Leykin, \emph{Stillman's conjecture via generic initial ideals} Comm.~Algebra {\bf 47} (2019), 2384--2395.
\bibitem{ESS} Daniel Erman, Steven V.~Sam, Andrew Snowden, \emph{Cubics in 10 variables vs. cubics in 1000 variables: uniformity phenomena for bounded degree polynomials}, Bull.~Amer.~Math. Soc.~(N.S.) {\bf 56} (2019), 87--114. 
\bibitem{Gar} Raymond Garver, \emph{The Tschirnhaus transformation}, The Annals of Mathematics, 2nd Ser., Vol. {\bf 29}, No. 1/4 (1927 -- 1928), pp.~329.
\bibitem{Ha} Mitsuyasu Hashimoto, \emph{Determinantal ideals without minimal free resolutions},
Nagoya Math. J.~Vol.~{\bf 118} (1990) 203--216.
\bibitem{HeHi} J\"urgen Herzog and Takayuki Hibi, \emph{Monomial ideals}, Springer-Verlag, London Limited, 2011.
\bibitem{Hul1} Heather Hulett, \emph{Maximum Betti numbers of homogeneous ideals with a given Hilbert function}, Comm.~Alg.~{\bf 21} (1993) 2335--2350. 
\bibitem{Hul2} Heather Hulett, \emph{A generalization of Macaulay's theorem}, Comm. Alg. 23 (1995), 1249--1263. MR1317399
\bibitem{HuMMS1} Craig Huneke, Paolo Mantero, Jason McCullough, and Alexandra Seceleanu, \emph{The projective dimension of codimension two algebras presented by quadrics}, 
J.~Algebra {\bf 393} (2013), pp.~170--186.  
\bibitem{HuMMS2} C.~Huneke, P.~Mantero, J.~McCullough, and A.~Seceleanu, \emph{A tight bound on the projective
dimension of four quardrics},  J.~Pure Appl.~Algebra {\bf 222} (2018), 2524--2551.
\bibitem{IP} Srikanth Iyerngar and Keith Pardue, \emph{Maximal minimal resolutions}, J.~Reine Angew.~Math. {\bf 512} (1999) 27--48.
\bibitem{Luc} \'Edouard Lucas, Th\'eorie des Nombres, Gauthier-Villars, Paris, 1891, Vol. 1.
\bibitem{LucSl} \'Edouard Lucas, Th\'eorie des Nombres, Gauthier-Villars, Paris, 1891, Vol. 1. [Scan of pages 488--499 only annotated by Neil J.~A.~Sloane] 
\bibitem{Mac} Francis Sowerby Macaulay, \emph{The algebraic theory of modular systems}, 1916. Cambridge Mathematical Library, 1994.
\bibitem{OEIS} Neil James Alexander Sloane,  editor, The On-Line Encyclopedia of Integer Sequences, published electronically at https://oeis.org/A000905 [October 14, 2018]
\bibitem{MaMc} Paolo Mantero and Jason McCullough, \emph{The projective dimension of three cubics is at most
5}, J.~Pure Appl.~Algebra {\bf 223} (2019), 1383--1410.
\bibitem{Mer} Jeffrey Mermin,  \emph{Lexicgraphic ideals}, Thesis, Cornell University, 2006.
\bibitem{MerP1} Jeffrey Mermin and Irena Peeva, \emph{Lexifying ideals}, Math.~Research Letters {\bf 13} (2006), 409--422.
\bibitem{MerP2} Jeffrey Mermin and Irena Peeva, \emph{Hilbert functions and lex ideals}, Journal of Algebra {\bf 313} (2007),  642--656.
\bibitem{Par1} Keith Pardue, \emph{Deformations of graded modules and connected loci on the Hilbert scheme}, in 
The Curves Seminar at Queen's, Vol.~XI, 131--149, Queen's Papers in Pure and Appl. Math., {\bf 105}, Queen's Univ., Kingston, ON, 1997.
\bibitem{Par2} Keith Pardue, \emph{Deformation classes of graded modules and maximal Betti numbers}, Illinois J.~Math. {\bf 40} (1996)  564--585. 
\bibitem{Sn1} Jan Snellman, \emph{Reverse lexicographic initial ideals of generic ideals are finitely generated}, Gr\"obner bases and applications 
(Linz, 1998), 504--518, London Math. Soc. Lecture Note Ser., {\bf 251}, Cambridge Univ. Press, Cambridge, 1998. 
\bibitem{Sn2} Jan Snellman, \emph{Gr\"obner bases and normal forms in a subring of the power series ring on countably many variables}, 
J.~Symbolic Comput.~{\bf 25} (1998), no. 3, 315--328. 
\bibitem{Sn3} Jan Snellman, \emph{Initial ideals of truncated homogeneous ideals}, Comm.~Algebra {\bf 26} (1998), no. 3, 813--824. 
\bibitem{SylH} James Joseph Sylvester and M.~James Hammond, \emph{On Hamilton's numbers}, Phil.~Trans.~Roy. Soc., 
{\bf 178} (1887), 285--312.
 
\end{thebibliography}
\end{document}